\documentclass[a4paper,reqno,10pt]{amsart}

\usepackage[
  colorlinks=true,
  linkcolor=blue,
  citecolor=blue,
  urlcolor=blue
]{hyperref}
\usepackage{a4wide}
\usepackage{amssymb,amstext,amsmath,amsthm,amsfonts,mathtools}
\usepackage{enumitem}
\usepackage{tikz-cd}
\setlist[enumerate,1]{label={\upshape(\arabic*)}}
\setlist[enumerate,2]{label={\upshape(\alph*)}}

\newtheorem{theorem}{Theorem}[section]
\newtheorem{proposition}[theorem]{Proposition}
\newtheorem{lemma}[theorem]{Lemma}

\newtheorem{conjecture}[theorem]{Conjecture}
\newtheorem*{theoremA}{Theorem A}
\theoremstyle{definition}
\newtheorem{remark}[theorem]{Remark}
\newtheorem*{organization}{Organization}
\newtheorem*{conventions}{Conventions and notation}
\newtheorem*{aiuse}{Use of AI}
\numberwithin{equation}{section}

\newcommand{\kfield}{k}
\newcommand{\Hom}{\operatorname{Hom}}
\newcommand{\rad}{\operatorname{rad}}
\newcommand{\Aut}{\operatorname{Aut}}
\newcommand{\im}{\operatorname{Im}}
\newcommand{\op}{\mathrm{op}}
\newcommand{\ee}{\mathrm e}

\title[The periodicity conjecture]
{A counterexample to the periodicity conjecture\\
for finite-dimensional algebras}

\author[H.\ Enomoto]{Haruhisa Enomoto}
\address{Parakeet Inc., Japan}
\email{haruhisa.enomoto.math@gmail.com}

\date{\today}

\begin{document}

\begin{abstract}
The periodicity conjecture asks whether a finite-dimensional algebra is
periodic whenever all its simple modules are periodic. We construct a
$36$-dimensional counterexample whose simple modules have period four.
The algebra has a $2$-dimensional nonperiodic module, which shows that
the algebra itself is not periodic.
\end{abstract}

\maketitle
\tableofcontents
\enlargethispage{5pt}

\section{Introduction}
\label{sec:introduction}

Let $A$ be a finite-dimensional algebra over a field $\kfield$. An
$A$-module $M$ is \emph{periodic} if $\Omega_A^n(M)\cong M$ for some $n>0$;
the least such $n$ is called the \emph{period} of $M$. Put
$A^{\ee}=A\otimes_\kfield A^{\op}$. The algebra $A$ is \emph{periodic} if
\[
 \Omega_{A^{\ee}}^n(A)\cong A
\]
as $A$-bimodules for some $n>0$. It is \emph{twisted periodic} if
\[
 \Omega_{A^{\ee}}^n(A)\cong {}_1A_\varphi
\]
for some $n>0$ and some $\varphi\in\Aut_\kfield(A)$. The bimodule
${}_1A_\varphi$ has the usual left action and right action
$x\cdot a=x\varphi(a)$. The periodicity conjecture asserts the following.

\begin{conjecture}[Periodicity conjecture {\cite[Introduction]{ChanDarpoIyamaMarczinzik}}]
\label{conj:periodicity}
Every finite-dimensional twisted periodic algebra over a field is periodic.
\end{conjecture}

For a finite-dimensional indecomposable algebra over an algebraically closed
field, Green--Snashall--Solberg \cite[Theorem~1.4]{GreenSnashallSolberg}
prove that twisted periodicity is equivalent to periodicity of every simple
module. Thus, in this setting, Conjecture~\ref{conj:periodicity} asks whether
periodicity of all simple modules implies periodicity of the algebra,
as formulated in \cite[p.~283]{ErdmannSkowronski}.

We give a counterexample to the periodicity conjecture. For each nonzero scalar $q$, we define an indecomposable algebra
$\Lambda(q)$ by an oriented triangle with two parallel arrows on each edge
and relations depending on $q$. Our main result is the following.

\begin{theoremA}[= Theorem~\ref{thm:main}]
Let $\kfield$ be an algebraically closed field containing an element $q$ of
infinite multiplicative order. The $36$-dimensional $\kfield$-algebra
$\Lambda(q)$ defined in Section~\ref{sec:algebra} has the following
properties: every simple right $\Lambda(q)$-module has period four, and
$\Lambda(q)$ is twisted periodic but not periodic. In particular, the
periodicity conjecture is false.
\end{theoremA}

The proof has two parts. We construct minimal projective resolutions showing
that the simple modules have period four. We then construct a two-dimensional
indecomposable nonprojective module whose even syzygies are pairwise
nonisomorphic. This module is not periodic, whereas every indecomposable
nonprojective module over a periodic algebra is periodic
\cite[Theorem~IV.11.19(ii)]{SkowronskiYamagata}. Hence $\Lambda(q)$ is not periodic.

\begin{remark}
The specialization $\Lambda(1)$ is the weighted surface algebra of
Erdmann--Skowro\'nski \cite[Example~5.11]{ErdmannSkowronskiSurface} with
weight $m=1$ and parameter $c=1$. It is symmetric and periodic of period
four \cite[Theorems~1.1 and~1.2]{ErdmannSkowronskiSurface}. The family
$\Lambda(q)$ is obtained by changing two coefficients in its defining
relations to $q$ and $q^{-1}$.

A related mechanism for nonperiodicity appears in Erdmann's construction
of modules over weakly symmetric special biserial algebras
\cite[Section~3]{ErdmannBounded}. She constructs families $M(\lambda)$, with $\lambda\ne0$,
satisfying $\Omega^2(M(\lambda))\cong M(v\lambda)$, where the nonzero
scalar $v$ is determined by the algebra relations and the chosen family.
When $v$ has infinite multiplicative order, successive even syzygies are
pairwise nonisomorphic. The two-dimensional modules constructed in
Section~\ref{sec:nonperiodic-module} satisfy the analogous formula with
$v=q$. Here this behavior occurs in an algebra whose simple modules have
period four.
\end{remark}

\begin{organization}
Section~\ref{sec:algebra} defines $\Lambda(q)$ and gives the path basis used
in the proof. Section~\ref{sec:simple-periodicity} constructs one period of
the minimal resolution of each simple module and deduces twisted periodicity
of $\Lambda(q)$. Section~\ref{sec:nonperiodic-module} constructs two-dimensional modules and
uses their syzygies to prove Theorem~\ref{thm:main}.
\end{organization}

\begin{conventions}
Throughout, $\kfield$ is an algebraically closed field. All modules are
finite-dimensional right modules. We omit the subscript from $\Omega$
when working with right $\Lambda(q)$-modules. Paths are multiplied in
traversal order: if $\alpha:i\to j$ and $\beta:j\to l$, then
$\alpha\beta$ first follows $\alpha$ and then $\beta$. Thus an arrow
$i\to j$ acts from $Me_i$ to $Me_j$ on a right module $M$. Indices are
elements of $\mathbb Z/3\mathbb Z$, represented by $1,2,3$.
\end{conventions}

\begin{aiuse}
The counterexample and its proof were found by GPT-6-Astra in research directed
by the author, and the first draft of the manuscript was written by
GPT-6-Astra. Claude Opus~5 reviewed the exposition of the manuscript. The
author set the structure and much of the wording of the present text. The
author is responsible for the result.
\end{aiuse}

\section{The algebra and its path basis}
\label{sec:algebra}

We define the algebra and determine a path basis. The dimensions and arrow
actions of its projective modules will be used to prove exactness of the
resolutions in Sections~\ref{sec:simple-periodicity} and
\ref{sec:nonperiodic-module}. Let $Q$ be the quiver
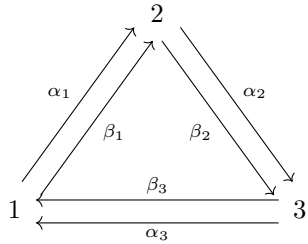
\[
\begin{tikzcd}[column sep=4em,row sep=6em,
  arrows={shorten <=1pt,shorten >=1pt}]
 & 2
   \arrow[dr,shift left=1ex,"\alpha_2"]
   \arrow[dr,shift right=1ex,"\beta_2"'] & \\
 1
   \arrow[ur,shift left=1ex,"\alpha_1"]
   \arrow[ur,shift right=1ex,"\beta_1"']
 && 3
   \arrow[ll,shift left=1ex,"\alpha_3"]
   \arrow[ll,shift right=1ex,"\beta_3"']
\end{tikzcd}
\]
Thus $\alpha_i,\beta_i:i\to i+1$. The construction and the resolutions
are defined for every $q\in\kfield^\times$; infinite order is needed
only to prove nonperiodicity. To write the relations cyclically while
placing the parameter at vertex $1$, set
\[
 \varepsilon_1=q,\qquad \varepsilon_2=\varepsilon_3=1.
\]
Define $\Lambda(q)=\kfield Q/I(q)$ by the following relations, for
$i\in\mathbb Z/3\mathbb Z$:
\begin{align}
 \alpha_i\alpha_{i+1}
 &=
 \varepsilon_i^{-1}
 \beta_i\alpha_{i+1}\beta_{i+2}\alpha_i\beta_{i+1},
 \label{eq:alpha-relation}\\
 \beta_i\beta_{i+1}
 &=
 \varepsilon_i
 \alpha_i\beta_{i+1}\alpha_{i+2}\beta_i\alpha_{i+1},
 \label{eq:beta-relation}\\
 \alpha_i\alpha_{i+1}\beta_{i+2}
 &=
 \beta_i\beta_{i+1}\alpha_{i+2}=0.
 \label{eq:zero-relations}
\end{align}

A path is called alternating if its arrows alternate between $\alpha$
and $\beta$. We describe a basis using these paths. For $1\leq l\leq6$, let
$p_{i,l}$ be the prefix of length $l$ of
\[
 \alpha_i\beta_{i+1}\alpha_{i+2}
 \beta_i\alpha_{i+1}\beta_{i+2},
\]
and let $p'_{i,l}$ be the prefix of length $l$ of
\[
 \beta_i\alpha_{i+1}\beta_{i+2}
 \alpha_i\beta_{i+1}\alpha_{i+2}.
\]

Both length-six paths are cycles at $i$. Put $\omega_i=p_{i,6}$.

\begin{proposition}
\label{prop:path-basis}
The set
\begin{equation}
 \{e_i,p_{i,1},\ldots,p_{i,5},
 p'_{i,1},\ldots,p'_{i,5},\omega_i\}
 \label{eq:projective-basis}
\end{equation}
is a basis of the right module
$P_i=e_i\Lambda(q)$. In particular, $\dim_\kfield\Lambda(q)=36$.
\end{proposition}

\begin{proof}
We divide the proof into two steps: first we show that the proposed paths
span $P_i$, then we construct a module that proves their independence.

\medskip
\noindent\textbf{Step 1: Spanning.}
Besides the zero relations \eqref{eq:zero-relations}, we have
\begin{equation}
 \alpha_i\beta_{i+1}\beta_{i+2}
 =\varepsilon_{i+1}\alpha_i\alpha_{i+1}\beta_{i+2}
   \alpha_i\beta_{i+1}\alpha_{i+2}=0,
 \label{eq:additional-zero}
\end{equation}
where the last equality uses \eqref{eq:zero-relations}. Likewise,
\[
 \beta_i\alpha_{i+1}\alpha_{i+2}
 =\varepsilon_{i+1}^{-1}\beta_i\beta_{i+1}\alpha_{i+2}
   \beta_i\alpha_{i+1}\beta_{i+2}=0.
\]
Together with \eqref{eq:zero-relations}, these equalities show that every
nonalternating path of length three containing both an $\alpha$-arrow and
a $\beta$-arrow is zero. Reducing the first two and the last two arrows of
$\alpha_i\alpha_{i+1}\alpha_{i+2}$ gives
\[
 \varepsilon_i^{-1}p'_{i,6}
 =\alpha_i\alpha_{i+1}\alpha_{i+2}
 =\varepsilon_{i+1}^{-1}p_{i,6},
\]
so the two length-six paths satisfy
\begin{equation}
 p'_{i,6}=\frac{\varepsilon_i}{\varepsilon_{i+1}}\omega_i.
 \label{eq:socle-paths}
\end{equation}
Also,
\[
 \beta_i\beta_{i+1}\beta_{i+2}
 =\varepsilon_i\alpha_i\beta_{i+1}\alpha_{i+2}
   \beta_i\alpha_{i+1}\beta_{i+2}=\varepsilon_i\omega_i.
\]
Thus
\begin{equation}
 \alpha_i\alpha_{i+1}\alpha_{i+2}
 =\varepsilon_{i+1}^{-1}\omega_i,
 \qquad
 \beta_i\beta_{i+1}\beta_{i+2}=\varepsilon_i\omega_i.
 \label{eq:triple-paths}
\end{equation}

A word of four $\alpha$-arrows is zero: reducing its first pair produces
an ending $\beta_{i+1}\alpha_{i+2}\alpha_i=0$. A word of four
$\beta$-arrows is also zero: reducing its first pair produces an ending
$\alpha_{i+1}\beta_{i+2}\beta_i=0$. Reducing the middle pair of
$\alpha_i\alpha_{i+1}\alpha_{i+2}\alpha_i$ gives
\[
 0=\alpha_i\alpha_{i+1}\alpha_{i+2}\alpha_i
   =\varepsilon_{i+1}^{-1}p_{i,6}\alpha_i.
\]
Reducing the middle pair of the word of four $\beta$-arrows gives
\[
 0=\beta_i\beta_{i+1}\beta_{i+2}\beta_i
   =\varepsilon_{i+1}p'_{i,6}\beta_i.
\]
Hence every alternating path of length at least seven is zero. A nonalternating
path of length at least four either contains
a nonalternating subpath of length three with both arrow types, or consists
entirely of one arrow type and contains a zero path of length four. Thus all such paths vanish. The only remaining
nonalternating paths have lengths two or three, and the relations replace
them by alternating paths of length five or six. This proves spanning.

\medskip
\noindent\textbf{Step 2: Linear independence.}
Let $W_i$ be a twelve-dimensional vector space with basis vectors labelled
by the paths in \eqref{eq:projective-basis}. In defining $W_i$, these labels
denote independent formal vectors, not elements of $P_i$. Place each vector
at the terminal vertex of its path, and let the vertex idempotents act
accordingly. Define the arrow actions by the following diagram and the two
formulas below it. The upper and lower rows correspond to paths starting
with $\alpha_i$ and $\beta_i$, respectively. An edge labelled $\gamma$
means right multiplication by $\gamma$; a coefficient $c$ on that edge
means that the image is $c$ times its target.
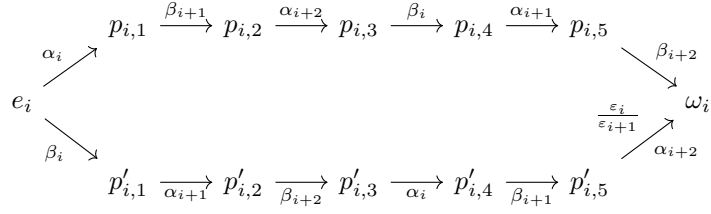
\[
\begin{tikzcd}[column sep=2em,row sep=1.4em]
 & p_{i,1} \arrow[r,"\beta_{i+1}"]
 & p_{i,2} \arrow[r,"\alpha_{i+2}"]
 & p_{i,3} \arrow[r,"\beta_i"]
 & p_{i,4} \arrow[r,"\alpha_{i+1}"]
 & p_{i,5} \arrow[dr,"\beta_{i+2}"] & \\
 e_i \arrow[ur,"\alpha_i"] \arrow[dr,"\beta_i"']
 & & & & & & \omega_i \\
 & p'_{i,1} \arrow[r,"\alpha_{i+1}"']
 & p'_{i,2} \arrow[r,"\beta_{i+2}"']
 & p'_{i,3} \arrow[r,"\alpha_i"']
 & p'_{i,4} \arrow[r,"\beta_{i+1}"']
 & p'_{i,5} \arrow[ur,"\alpha_{i+2}"',
   "\frac{\varepsilon_i}{\varepsilon_{i+1}}"] &
\end{tikzcd}
\]
The two additional nonzero arrow actions are
\begin{equation}
 p_{i,1}\alpha_{i+1}=\varepsilon_i^{-1}p'_{i,5},
 \qquad p'_{i,1}\beta_{i+1}=\varepsilon_i p_{i,5}.
 \label{eq:additional-actions}
\end{equation}

Every other arrow action is zero. We verify that these actions satisfy
the defining relations, so that $W_i$ is a $\Lambda(q)$-module.

On $e_i$, relations
\eqref{eq:alpha-relation} and \eqref{eq:beta-relation} are precisely
\eqref{eq:additional-actions}. On the basis vectors of length one, the
only instances of these two relations with nonzero values on either side give
\begin{align*}
 (p_{i,1}\alpha_{i+1})\alpha_{i+2}
 &=\varepsilon_{i+1}^{-1}\omega_i
 =p_{i,1}(\varepsilon_{i+1}^{-1}p'_{i+1,5}),\\
 (p'_{i,1}\beta_{i+1})\beta_{i+2}
 &=\varepsilon_i\omega_i
 =p'_{i,1}(\varepsilon_{i+1}p_{i+1,5}).
\end{align*}
The remaining instances on $p_{i,1}$ and $p'_{i,1}$ have both sides
zero: we have
$p_{i,1}\beta_{i+1}\beta_{i+2}=0$ and
$p'_{i,1}\alpha_{i+1}\alpha_{i+2}=0$, and their paths of length five
also act by zero. On a basis vector of length at least two, the only
nonzero arrow actions follow the alternating paths in the diagram.
Thus two consecutive $\alpha$-arrows or two consecutive $\beta$-arrows
act by zero, and every path of length five acts by zero. This verifies
\eqref{eq:alpha-relation} and \eqref{eq:beta-relation}.

For the zero relations \eqref{eq:zero-relations}, their actions on $e_i$
are zero because $p'_{i,5}\beta_{i+2}=0$ and
$p_{i,5}\alpha_{i+2}=0$. On a basis vector of length one, the first two
arrows either act by zero or give a multiple of $\omega_i$, which the
third arrow kills. On a basis vector of length at least two, the first
two arrows act by zero. Hence all the defining relations hold on $W_i$.

Consequently $W_i$ is a $\Lambda(q)$-module. It is generated by $e_i$,
so evaluation $P_i\to W_i$, $x\mapsto e_i\cdot x$, is surjective.
Spanning and $\dim_\kfield W_i=12$ now give
$12\leq\dim_\kfield P_i\leq12$, so the spanning set
\eqref{eq:projective-basis} is a basis. Since
$\Lambda(q)=P_1\oplus P_2\oplus P_3$, we obtain
$\dim_\kfield\Lambda(q)=3\cdot12=36$.
\end{proof}

The ideal of relations is contained in the square of the arrow ideal
and contains every path of length seven, so it is admissible. Since
the quiver is connected, $\Lambda(q)$ is indecomposable as an algebra,
and admissibility also shows that
$P_i$ is the indecomposable projective module at vertex $i$, with
one-dimensional simple top $S_i$.

\section{Periodicity of the simple modules}
\label{sec:simple-periodicity}

We construct one period of the minimal projective resolution of $S_i$.
Fix $i$ and put
\[
 r_i=\beta_{i+1}\alpha_{i+2}\beta_i\alpha_{i+1},
 \qquad
 s_i=\alpha_{i+1}\beta_{i+2}\alpha_i\beta_{i+1},
\]
so that relations \eqref{eq:alpha-relation} and
\eqref{eq:beta-relation} become
$\beta_i\beta_{i+1}=\varepsilon_i\alpha_i r_i$ and
$\alpha_i\alpha_{i+1}=\varepsilon_i^{-1}\beta_i s_i$.
We define homomorphisms of right modules
\[
 d_1:P_{i+1}^{\oplus2}\longrightarrow P_i,\qquad
 d_2:P_{i+2}^{\oplus2}\longrightarrow P_{i+1}^{\oplus2},\qquad
 d_3:P_i\longrightarrow P_{i+2}^{\oplus2}
\]
by left multiplication on column vectors by the matrices
\[
 d_1=
 \begin{bmatrix}\alpha_i&\beta_i\end{bmatrix},
 \qquad
 d_2=
 \begin{bmatrix}
  \alpha_{i+1}&-\varepsilon_i r_i\\
  -\varepsilon_i^{-1}s_i&\beta_{i+1}
 \end{bmatrix},
 \qquad
 d_3=
 \begin{bmatrix}
  \alpha_{i+2}\\
  (\varepsilon_i\varepsilon_{i+1})^{-1}\beta_{i+2}
 \end{bmatrix}.
\]
Write $d_j=d_j^{(i)}$ while $i$ is fixed. Let $\pi_i:P_i\to S_i$ be the quotient map, and let
$\iota_i:S_i\to P_i$ send $\pi_i(e_i)$ to $\omega_i$. The latter
is a module map because $\omega_i$ is a cycle killed by every arrow.

\begin{proposition}
\label{prop:simple-resolution}
For each $i$, the sequence
\[
\begin{tikzcd}[column sep=2.8em]
0 \arrow[r] &
S_i \arrow[r,"\iota_i"] &
P_i \arrow[r,"d_3^{(i)}"] &
P_{i+2}^{\oplus2} \arrow[r,"d_2^{(i)}"] &
P_{i+1}^{\oplus2} \arrow[r,"d_1^{(i)}"] &
P_i \arrow[r,"\pi_i"] &
S_i \arrow[r] &
0
\end{tikzcd}
\]
is exact and minimal. Consequently, $S_i$ has period four.
\end{proposition}

\begin{proof}
We show that consecutive maps compose to zero and that $d_1,d_2,d_3$
have ranks $11,13,11$, respectively. These ranks will give exactness
at the two middle projectives by comparison with the kernel dimensions.
Relations \eqref{eq:alpha-relation} and \eqref{eq:beta-relation}
at $i$ give $d_1d_2=0$.
For the other product, the identities
\[
 \alpha_{i+1}\alpha_{i+2}
 =\varepsilon_{i+1}^{-1}r_i\beta_{i+2},
 \qquad
 \beta_{i+1}\beta_{i+2}
 =\varepsilon_{i+1}s_i\alpha_{i+2}
\]
give $d_2d_3=0$.

The paths of positive length in \eqref{eq:projective-basis} span
$\rad P_i$ and all start with $\alpha_i$ or $\beta_i$. Thus
$\im d_1=\rad P_i$ has dimension $11$, and hence
$\dim_\kfield\ker d_1=24-11=13$.

We next prove $\im d_2=\ker d_1$ by exhibiting thirteen independent
vectors in $\im d_2$. Let $v,w$ be the columns of $d_2$.
For $1\leq l\leq5$, multiplication gives
\[
 v p'_{i+2,l}=\binom{p_{i+1,l+1}}{0},
 \qquad
 w p_{i+2,l}=\binom{0}{p'_{i+1,l+1}}.
\]
These ten vectors are independent. We adjoin $v,w$ and the vector
\[
 v\alpha_{i+2}
 =\binom{\varepsilon_{i+1}^{-1}p'_{i+1,5}}
            {-\varepsilon_i^{-1}p_{i+1,5}}.
\]
In a linear relation among these thirteen vectors, the coefficient of
$\alpha_{i+1}$ in the first coordinate forces the coefficient of $v$ to
be zero. The coefficient of $\beta_{i+1}$ in the second coordinate does
the same for $w$. The coefficient of $p'_{i+1,5}$ in the first coordinate
then forces the coefficient of $v\alpha_{i+2}$ to be zero. Independence
of the original ten vectors proves $\dim_\kfield\im d_2\geq13$.
Since $\im d_2\subseteq\ker d_1$ and $\dim_\kfield\ker d_1=13$,
we obtain $\im d_2=\ker d_1$.

It follows that $\dim_\kfield\ker d_2=24-13=11$. To prove
$\im d_3=\ker d_2$ and $\ker d_3=\im\iota_i$, we compute $d_3$
on the path basis of $P_i$. Put
$\kappa_i=(\varepsilon_i\varepsilon_{i+1})^{-1}$. The values of $d_3$
on the path basis are
\begin{align*}
 d_3(e_i)&=\binom{\alpha_{i+2}}{\kappa_i\beta_{i+2}},&
 d_3(\omega_i)&=0,\\
 d_3(p_{i,1})&=\binom{\varepsilon_{i+2}^{-1}p'_{i+2,5}}
                         {\kappa_i p'_{i+2,2}},&
 d_3(p'_{i,1})&=\binom{p_{i+2,2}}
                          {\kappa_i\varepsilon_{i+2}p_{i+2,5}},\\
 d_3(p_{i,l})&=\binom{0}{\kappa_i p'_{i+2,l+1}},&
 d_3(p'_{i,l})&=\binom{p_{i+2,l+1}}{0}\quad(2\leq l\leq5).
\end{align*}
In a linear relation among the eleven nonzero images, comparing the
coefficient of $\alpha_{i+2}$ in the first coordinate forces the coefficient
of $d_3(e_i)$ to be zero. The coefficients of $p'_{i+2,2}$ in the second
coordinate and $p_{i+2,2}$ in the first then force the coefficients of
$d_3(p_{i,1})$ and $d_3(p'_{i,1})$ to be zero. The remaining eight vectors are nonzero multiples of
distinct coordinate basis vectors, so they are independent. The image
of $\omega_i$ vanishes because paths of length seven vanish. Therefore
\[
 \ker d_3=\kfield\omega_i=\im\iota_i,
 \qquad \dim_\kfield\im d_3=11=\dim_\kfield\ker d_2.
\]
Together with $\im d_1=\rad P_i=\ker\pi_i$, this proves exactness.

All entries of the differentials belong to the radical, so the sequence is minimal. It gives $\Omega^4(S_i)\cong S_i$;
the first three syzygies have dimensions $11,13,11$, whereas
$\dim_\kfield S_i=1$. Thus the period is four.
\end{proof}

To deduce twisted periodicity from the simple resolutions, we use the
following result of Green--Snashall--Solberg.

\begin{theorem}[{\cite[Theorem~1.4 and its proof]{GreenSnashallSolberg}}]
\label{thm:gss}
Let $A$ be a finite-dimensional indecomposable algebra over an algebraically
closed field. Suppose that every simple right $A$-module is periodic, and
let $n$ be the least positive integer such that $\Omega_A^n(S)\cong S$
for every simple right $A$-module $S$. Then
\[
 \Omega_{A^{\ee}}^n(A)\cong {}_1A_\varphi
\]
for some $\varphi\in\Aut_\kfield(A)$.
\end{theorem}

Since $\Lambda(q)$ is indecomposable and the common period of its simple
modules is four, Theorem~\ref{thm:gss} gives
\begin{equation}
 \Omega^4_{\Lambda(q)^{\ee}}(\Lambda(q))
 \cong{}_1\Lambda(q)_\varphi
 \label{eq:twisted-period}
\end{equation}
for some $\varphi\in\Aut_\kfield(\Lambda(q))$. Thus $\Lambda(q)$ is
twisted periodic. To refute Conjecture~\ref{conj:periodicity}, it remains
to show that $\Lambda(q)$ is not periodic.

\section{Nonperiodic modules}
\label{sec:nonperiodic-module}

We construct indecomposable nonprojective modules whose even syzygies are
pairwise nonisomorphic when $q$ has infinite order. Every indecomposable
nonprojective module over a periodic algebra is periodic
\cite[Theorem~IV.11.19(ii)]{SkowronskiYamagata}, so any one of these modules
will show that $\Lambda(q)$ is not periodic. For $\lambda\in\kfield^\times$,
define $U_\lambda$ by
\[
\begin{tikzcd}[column sep=5em]
U_\lambda e_1=\kfield
 \arrow[r,shift left=1ex,"\lambda"]
 \arrow[r,shift right=1ex,"1"']
&
U_\lambda e_2=\kfield,
\end{tikzcd}
\qquad
U_\lambda e_3=0.
\]
The two maps are the actions of $\alpha_1$ and $\beta_1$, respectively, and
all other arrows act as zero. Every path of length two acts by zero,
so all the defining relations hold. Thus $U_\lambda$ is a two-dimensional
$\Lambda(q)$-module. Let $u_1,u_2$ denote the vectors $1$ in its vertex
spaces at $1,2$.
A decomposition into two nonzero direct summands would place these
two one-dimensional vertex spaces in different summands, forcing
$\beta_1$ to act by zero. Thus $U_\lambda$ is indecomposable. It is
nonprojective, since every nonzero projective module has dimension at least
$12$ by Proposition~\ref{prop:path-basis}.

\begin{proposition}
\label{prop:module-resolution}
For every $\lambda\in\kfield^\times$, there is an exact minimal sequence
\[
\begin{tikzcd}[column sep=3.4em]
0 \arrow[r] &
U_{q\lambda} \arrow[r,"\jmath_\lambda"] &
P_2 \arrow[r,"\alpha_1-\lambda\beta_1"] &
P_1 \arrow[r,"\pi_\lambda"] &
U_\lambda \arrow[r] &
0 .
\end{tikzcd}
\]
The map $P_2\to P_1$ is left multiplication by
$\alpha_1-\lambda\beta_1$, and $\pi_\lambda(e_1)=u_1$. On $U_{q\lambda}$, the map
$\jmath_\lambda$ sends $1\in U_{q\lambda}e_1$ to
$z=\beta_2\beta_3+q\lambda\alpha_2\alpha_3$ and
$1\in U_{q\lambda}e_2$ to $\omega_2$.
Consequently $\Omega^2(U_\lambda)\cong U_{q\lambda}$, and iteration gives
\begin{equation}
 \Omega^{2m}(U_\lambda)\cong U_{q^m\lambda}
 \qquad (m\geq0).
 \label{eq:syzygy-drift}
\end{equation}
\end{proposition}

\begin{proof}
Write $\pi=\pi_\lambda$, and let $h:P_2\to P_1$ be left
multiplication by $\alpha_1-\lambda\beta_1$. We first determine
$\ker\pi=\im h$, then identify $\ker h$ with $U_{q\lambda}$.

\medskip
\noindent\textbf{Step 1: The first projective cover.}
Since $u_1\beta_1=u_2$, the vector $u_1$ generates $U_\lambda$, so
$\pi$ is surjective and $\dim_\kfield\ker\pi=12-2=10$. Since
$\pi(\alpha_1-\lambda\beta_1)=\lambda u_2-\lambda u_2=0$,
we have $\im h\subseteq\ker\pi$. To prove equality, it suffices to
show that $\im h$ has dimension ten. The path basis gives
\begin{align*}
 h(e_2)&=\alpha_1-\lambda\beta_1,\\
 h(p_{2,1})&=q^{-1}p'_{1,5}-\lambda p'_{1,2},&
 h(p'_{2,1})&=p_{1,2}-\lambda q p_{1,5},\\
 h(p_{2,l})&=-\lambda p'_{1,l+1},&
 h(p'_{2,l})&=p_{1,l+1}\quad(2\leq l\leq4),\\
 h(p_{2,5})&=-\lambda q\omega_1,&
 h(p'_{2,5})&=\omega_1,\qquad h(\omega_2)=0.
\end{align*}
The six images with $2\leq l\leq4$ are nonzero multiples of the six
basis vectors of lengths three through five. Adjoin the images of
$e_2,p_{2,1},p'_{2,1},p'_{2,5}$. In a linear relation among these ten
vectors, comparing the coefficients of $\alpha_1,p'_{1,2},p_{1,2}$ and
$\omega_1$ forces the four added coefficients to be zero. Thus all ten
are independent. The table shows that they span $\im h$, so
$\dim_\kfield\im h=10$.
Thus $\im h=\ker\pi$. Since $\im h\subseteq\rad P_1$, the
surjection $\pi$ is a projective cover.

\medskip
\noindent\textbf{Step 2: The second syzygy.}
We first find a basis of the vector space $\ker h$. The element $z$ defined
in the statement can be written in the path basis as
\[
 z=\beta_2\beta_3+q\lambda\alpha_2\alpha_3
   =p_{2,5}+q\lambda p'_{2,5}.
\]
The two length-five paths are distinct basis vectors, so $z$ and
$\omega_2$ are independent. The displayed images give $h(z)=0$:
\[
 h(z)=-\lambda q\omega_1+q\lambda\omega_1=0.
\]
Also $h(\omega_2)=0$, and
$\dim_\kfield\ker h=12-10=2$. Therefore
$\ker h=\kfield z\oplus\kfield\omega_2$.
To identify this kernel as a module, we compute its vertex spaces and
arrow actions. The vertex positions are $z\in P_2e_1$ and
$\omega_2\in P_2e_2$.
By \eqref{eq:zero-relations} and \eqref{eq:triple-paths} at vertex $2$,
\begin{align*}
 z\alpha_1
 &=\beta_2\beta_3\alpha_1+q\lambda\alpha_2\alpha_3\alpha_1
 =q\lambda\omega_2,\\
 z\beta_1
 &=\beta_2\beta_3\beta_1+q\lambda\alpha_2\alpha_3\beta_1
 =\omega_2.
\end{align*}
Every arrow kills $\omega_2$, because paths of length seven vanish.
These actions agree with those defining $U_{q\lambda}$. Since $z,\omega_2$
form a basis of $\ker h$, the map $\jmath_\lambda$ is an isomorphism
from $U_{q\lambda}$ onto $\ker h$. Since $\ker h\subseteq\rad P_2$,
the surjection $P_2\to\ker\pi$ induced by $h$ is also a projective
cover. Hence $\Omega^2(U_\lambda)\cong U_{q\lambda}$, and iteration proves
\eqref{eq:syzygy-drift}.
\end{proof}

To show that the even syzygies are pairwise nonisomorphic, we use the
following vanishing for distinct parameters.

\begin{lemma}
\label{lem:hom-separation}
If $\lambda,\mu\in\kfield^\times$ and $\lambda\ne\mu$, then
$\Hom_{\Lambda(q)}(U_\lambda,U_\mu)=0$.
\end{lemma}

\begin{proof}
A homomorphism is given by scalars $s_1,s_2$ at vertices $1,2$.
Compatibility with the actions of $\beta_1$ and $\alpha_1$ gives
$s_1=s_2$ and $\lambda s_2=\mu s_1$, respectively. Since $\lambda\ne\mu$,
both scalars are zero.
\end{proof}

When $q$ has infinite order and $\lambda\ne0$, the scalars
$q^m\lambda$ for $m\geq0$ are pairwise distinct. Proposition~\ref{prop:module-resolution} and
Lemma~\ref{lem:hom-separation} give the following sequence of pairwise
nonisomorphic modules, with each term the second syzygy of the preceding one:
\[
\begin{tikzcd}[column sep=3.7em]
U_\lambda
 \arrow[r,mapsto,"\Omega^2"] &
U_{q\lambda}
 \arrow[r,mapsto,"\Omega^2"] &
U_{q^2\lambda}
 \arrow[r,mapsto,"\Omega^2"] &
\cdots .
\end{tikzcd}
\]

We now combine the path basis, the simple resolutions, and this sequence
to prove the theorem announced in the introduction.

\begin{theorem}
\label{thm:main}
Let $\kfield$ be algebraically closed, and let $q\in\kfield^\times$ have
infinite multiplicative order. Then $\Lambda(q)$ has dimension $36$, and
every simple right $\Lambda(q)$-module has period four. The algebra
$\Lambda(q)$ is twisted periodic but not periodic. In particular, the
periodicity conjecture is false.
\end{theorem}

\begin{proof}
Propositions~\ref{prop:path-basis} and \ref{prop:simple-resolution}
give the dimension and the periods of the simple modules, and
\eqref{eq:twisted-period} gives twisted periodicity. It remains to prove
that $\Lambda(q)$ is not periodic. Fix
$\lambda\in\kfield^\times$. If
$U_\lambda$ were periodic, say
$\Omega^p(U_\lambda)\cong U_\lambda$ for $p>0$, then
\[
 U_\lambda\cong\Omega^{2p}(U_\lambda)\cong U_{q^p\lambda}
\]
by Proposition~\ref{prop:module-resolution}. This is impossible: since $q$ has infinite order and $\lambda\ne0$,
$q^p\lambda\ne\lambda$, and Lemma~\ref{lem:hom-separation} gives
\[
 \Hom_{\Lambda(q)}(U_{q^p\lambda},U_\lambda)=0.
\]
Thus $U_\lambda$ is not periodic. Since $U_\lambda$ is indecomposable
and nonprojective, \cite[Theorem~IV.11.19(ii)]{SkowronskiYamagata}
implies that $\Lambda(q)$ is not periodic. Hence $\Lambda(q)$ is
a counterexample to Conjecture~\ref{conj:periodicity}.
\end{proof}

\end{document}